\documentclass[11pt]{amsart}
\usepackage{amssymb,latexsym,amsmath,amsthm,amsfonts}
\usepackage[all,cmtip]{xy}

\begin{document}

\vspace{0.2in}

\title[Cuspidality Criterion]
{A Cuspidality Criterion for the Exterior Square Transfer of Cusp Forms on ${\rm GL}(4)$}
\author[M. Asgari \& A. Raghuram]{Mahdi Asgari \ \ and \ \  A. Raghuram}

\maketitle

\begin{center}{\it Dedicated to Freydoon Shahidi on the occasion of his sixtieth birthday.}\end{center}

\begin{abstract}
For a cuspidal automorphic representation $\Pi$ of ${\rm GL}(4,{\mathbb A})$,
H. Kim proved that the exterior square transfer $\wedge^2\Pi$
is an isobaric automorphic representation of ${\rm GL}(6,{\mathbb A})$. In this paper
we characterize those representations $\Pi$ for which $\wedge^2\Pi$ is 
cuspidal. 
\end{abstract}

\numberwithin{equation}{section}
\newtheorem{theorem}[equation]{Theorem}
\newtheorem{cor}[equation]{Corollary}
\newtheorem{lemma}[equation]{Lemma}
\newtheorem{proposition}[equation]{Proposition}
\newtheorem{con}[equation]{Conjecture}
\newtheorem{ass}[equation]{Assumption}
\newtheorem{defn}[equation]{Definition}
\newtheorem{exer}[equation]{Exercise}

\theoremstyle{remark}
\newtheorem{rem}[equation]{Remark}
\newtheorem{exam}[equation]{Example}

\newcommand{\C}{\mathbb C}
\newcommand{\rl}{\mathbb R}
\newcommand{\Ad}{\mathbb A}

\renewcommand{\aa}{\alpha}
\renewcommand{\a}[1]{\ensuremath{{\alpha_{#1}}}}
\newcommand{\ba}[1]{\ensuremath{{\overline{\alpha}_{#1}}}}
\newcommand{\bb}{\beta}
\renewcommand{\b}[1]{\ensuremath{{\beta_{#1}}}}
\newcommand{\e}{\epsilon}
\newcommand{\g}{\gamma}

\newcommand{\As}{\mathrm{As}}

\newcommand{\GL}{\mathrm{GL}}
\newcommand{\GSp}{\mathrm{GSp}}
\newcommand{\GSpin}{\mathrm{GSpin}}
\newcommand{\PSL}{\mathrm{PSL}}
\newcommand{\SL}{\mathrm{SL}}
\newcommand{\SO}{\mathrm{SO}}
\newcommand{\SSp}{\mathrm{Sp}}

\newcommand{\w}[1]{\ensuremath{{#1}^\vee}}

\bigskip

\section{Introduction and statement of the main theorem}
\label{sec:intro}

Let $F$ be a number field whose ad\`ele ring we denote by ${\mathbb A}_F$. 
Let $G_1$ and $G_2$ be two connected reductive linear algebraic groups over $F$ and let $^LG_1$ and $^LG_2$ 
be the corresponding $L$-groups. Given an $L$-homomorphism $r : {}^LG_1 \to {}^LG_2$, Langlands principle of 
functoriality predicts the existence of a transfer $\Pi \mapsto r(\Pi)$ of the $L$-packet of 
an automorphic representation $\Pi$ of $G_1({\mathbb A}_F)$ to an $L$-packet $r(\Pi)$ of 
automorphic representations of $G_2({\mathbb A}_F)$. Now assume that $G_2$ is a general linear group. 
We note that an $L$-packet for a general linear group is a singleton set. 
For applications of functoriality one needs to understand the image and fibers of the correspondence 
$\Pi \mapsto r(\Pi)$. In particular, it is necessary to understand what conditions on $\Pi$ ensure that the transfer $r(\Pi)$ is cuspidal.  
 
The main aim of this paper is to describe a cuspidality criterion for the 
transfer of automorphic representations from ${\rm GL}(4,{\mathbb A}_F)$ to those of ${\rm GL}(6,{\mathbb A}_F)$ 
corresponding to the exterior square map $\wedge^2 : {\rm GL}(4,{\mathbb C}) \to {\rm GL}(6,{\mathbb C})$. 
Langlands functoriality in this case is a deep theorem due to H.~Kim \cite{kim-jams}.

Roughly speaking our main result states that if $\Pi$ is a cuspidal automorphic representation of 
${\rm GL}(4,{\mathbb A}_F)$, then its exterior square transfer $\wedge^2\Pi$ is not cuspidal if and only if 
$\Pi$ is a transfer of a representation from a `smaller' group such as ${\rm GSpin}(4)$, ${\rm GSp}(4)$, or 
${\rm GL}(2)$ over a quadratic extension $E/F$.  It is easy to verify that if $\Pi$ is 
a transfer from one such `smaller' group, then $\wedge^2\Pi$ is not cuspidal. It is the converse 
which is more difficult to prove. 
Another equivalent criterion, which is easier to state, is that $\wedge^2\Pi$ is not cuspidal if and only if 
$\Pi$ is essentially self-dual (i.e., $\Pi \cong \Pi^{\vee}\otimes\chi$ for some Hecke character
$\chi$ of $F$) or it has a nontrivial self-twist (i.e., $\Pi \cong \Pi\otimes\chi$ for a nontrivial Hecke character
$\chi$ of $F$), except that in the essentially self-dual case we need to exclude a particular kind of representation. 
We are not able to satisfactorily 
explain this exception purely in automorphic terms or in terms of $L$-functions. However, it admits a 
nice explanation in terms of Galois representations as we explain after stating the main theorem. 
(See also Section \ref{sec:galois} below.)

Let $\Pi=\otimes_v \Pi_v$ and $\Sigma=\otimes_v \Sigma_v$ 
be irreducible isobaric automorphic representations of $\GL(4,\Ad_F)$ and $\GL(6,\Ad_F)$, respectively. 
Assume that $S$ is a finite set of places of $F$, including all the archimedean ones, outside of which 
both of the representations are unramified. We say $\Sigma$ is an exterior square transfer of $\Pi$ 
if for all $v\not\in S$ we have $\Sigma_v = \wedge^2(\Pi_v)$, i.e., the semi-simple conjugacy class 
in $\GL(6,\C)$ determining $\Sigma_v$ is generated by the image under $\wedge^2$ of the semi-simple 
conjugacy class in $\GL(4,\C)$ determining $\Pi_v$. By strong multiplicity one theorem 
(see Theorem \ref{thm:isobaric} below) such a $\Sigma$ would be unique. We will denote it by $\wedge^2\Pi$. 
The existence of $\wedge^2\Pi$ was established by H. Kim \cite[Theorem A]{kim-jams}. Furthermore, 
he showed that if $\Pi_v$ is not supercuspidal for the places $v$ dividing $2$ or $3$, then the 
local component $\Pi_v$ and $(\wedge^2\Pi)_v$ are compatible via the local Langlands correspondence for 
$\GL(4,F_v)$ and $\GL(6,F_v)$. The assumption at $v\vert 2,3$ was made 
because of complications posed by supercuspidal representations, especially of $\GL(4,F_v)$. 
In any event, it has no bearing on our result as we 
do not need the fact that the local components of $\wedge^2\Pi$ and $\Pi$ are compatible via 
the local Langlands correspondence at all places. We now state the main theorem of this article.

\begin{theorem} 
\label{thm:main}
Let $F$ be a number field and let $\Pi$ be a cuspidal automorphic representation of $\GL(4,\Ad_F)$. 
The following are equivalent: 

\begin{itemize}

\item[{\bf (i)}] $\wedge^2\Pi$ is not cuspidal. 

\item[{\bf (ii)}] $\Pi$ is one of the following: 

	\begin{itemize}

	\item[(a)] $\Pi=\pi_1 \boxtimes \pi_2$, a transfer from $\GL(2,\Ad_F) \times 
	\GL(2,\Ad_F)$ to $\GL(4,\Ad_F)$ via the automorphic tensor product $\boxtimes$.
	(This may also be viewed as the transfer from the split $\GSpin(4)$ to ${\rm GL}(4)$.)

	\item[(b)] $\Pi=\As(\pi)$, the Asai transfer 
	of a dihedral cuspidal automorphic representation $\pi$ of $\GL(2,\Ad_E)$
	where $E/F$ is a quadratic extension. 
	(This may also be viewed as the transfer to ${\rm GL}(4)$ from the quasi-split non-split 
	${\rm GSpin}^*(4)$ over $F$ which splits over $E$.) 

	\item[(c)] $\Pi$ is the functorial transfer of a cuspidal representation $\pi$ of 
	$\GSp(4,\Ad_F)$ associated with the natural embedding 
	of the dual group $\GSp(4,\C)$ into $\GL(4,\C)$.  The representation $\pi$ may be taken to be 
	globally generic. 

	\item[(d)] $\Pi={\rm I}_E^F(\pi)$, the automorphic induction of a cuspidal
	automorphic representation $\pi$ of $\GL(2,\Ad_E)$, where $E/F$ is a 
	quadratic extension. 

	\end{itemize}

\item[{\bf (iii)}] $\Pi$ satisfies one of the following:

\begin{itemize}

	\item[($\alpha$)]
	$\Pi \cong \w{\Pi} \otimes \chi$ for some Hecke character $\chi$ of $F$, and $\Pi$ is not the 
        Asai transfer of a nondihedral cuspidal representation. 

	\item[($\beta$)] $\Pi \cong \Pi \otimes \chi$ for a nontrivial Hecke character $\chi$ of $F$.

	\end{itemize}

\end{itemize}

\end{theorem}

We now briefly sketch the proof of the theorem. 
As mentioned above, it is easy to verify that {\bf (ii)} implies both {\bf (i)} and {\bf (iii)}. 
In Section \ref{sec:2-implies-1} we explicitly write down the isobaric decomposition of 
$\wedge^2\Pi$ for each of the cases {\bf (ii)}(a)--{\bf (ii)}(d). In order to check that two 
isobaric representations are isomorphic we repeatedly use a strong multiplicity one 
theorem, due to Jacquet and Shalika, recalled in Section \ref{sec:j-ps-s}.

The proof of {\bf (i)} $\Longrightarrow$ {\bf (iii)}, described in Section \ref{sec:1-implies-3}, uses some details from 
the Langlands--Shahidi machinery. We show in Proposition~\ref{prop:gl1-gl3-twist} that if $\Pi$ is not essentially
self-dual, then $\wedge^2\Pi$ cannot have a degree $1$ or degree $3$ isobaric summand. In 
Proposition~\ref{prop:gl2-twist} we verify that if $\Pi$ does not admit a nontrivial self-twist, then 
$\wedge^2\Pi$ cannot have a degree $2$ isobaric summand. 
For a cuspidal representation $\sigma$ of ${\rm GL}(m,{\mathbb A}_F)$ consider 
the Langlands $L$-function $L(s, \Pi \times \sigma, \wedge^2 \otimes\rho_m)$, where $\rho_m$ is the standard
representation of ${\rm GL}(m,{\mathbb C})$. These $L$-functions appear in the Langlands--Shahidi
machinery for a particular choice of Levi subgroup when the ambient group is ${\rm GSpin}(2m+6)$. 
We use Kim \cite[\S3]{kim-jams} to show that, under the above mentioned hypothesis on $\Pi$, these 
(partial) $L$-functions are holomorphic at $s=1$. The implication then follows from a well-known 
result of Jacquet, Piatetski-Shapiro, and Shalika recalled in Section \ref{sec:j-ps-s}. We summarize some
of the preliminaries we need from the Langlands--Shahidi machinery in Section \ref{sec:ls-theory}. 

Finally, the proof of {\bf (iii)} $\Longrightarrow$ {\bf (ii)} is given in Section \ref{sec:3-implies-2}. 
It relies on the so-called `descent theory' for classical groups. As a general reference for descent theory we refer 
to D. Soudry's exposition \cite{soud}.

In Section \ref{sec:examples} we present a few examples, some only conjectural, illustrating the above theorem. 
In Section \ref{sec:galois} we ask the question whether it is possible to see the cuspidality criterion 
from the `Galois side'. 
The question can be made precise based on the philosophy that there is a 
correspondence between automorphic representations $\pi$ of ${\rm GL}(n,{\mathbb A}_F)$
and $\ell$-adic $n$-dimensional representations $\sigma$ of the absolute Galois group of $F$. 
Let us denote this correspondence by $\pi \mapsto \sigma(\pi)$. Part of this philosophy is that $\pi$ is 
supposed to be cuspidal if and only if $\sigma(\pi)$ is irreducible. We refer to 
Ramakrishnan \cite{ramak-irr-cusp} 
for the state of the art on this issue. 
In view of the above theorem one can ask the following question. Let $\sigma$ be a four-dimensional 
irreducible Galois representation; what condition on $\sigma$ will ensure that $\wedge^2\sigma$ is irreducible? 
Upon posing this question in a talk at the Oklahoma Joint Automorphic Forms Seminar,  
A. Kable came up with a very elegant theorem which reflects the equivalence of {\bf (i)} and {\bf (iii)} 
in Theorem~\ref{thm:main}. We are grateful to him for allowing us to include his theorem and its proof 
in Section \ref{sec:galois}. Recall that in {\bf (iii)}($\alpha$) of Theorem~\ref{thm:main} above, we had to 
exclude the Asai transfer of a nondihedral cuspidal representation if $\Pi$ is essentially self-dual. 
On the Galois side, this is reflected in the fact that if a four-dimensional irreducible representation 
$\sigma$ is essentially self-dual of orthogonal type, then for $\wedge^2\sigma$ to be reducible 
the image of $\sigma$ should lie in the connected component of the identity in the algebraic 
group ${\rm GO}(4)$; see Theorem~\ref{thm:galois-side}.

Cuspidality criteria are important not only for their intrinsic value in helping us better understand a given 
instance of functoriality but also because they have important arithmetic applications. 
D. Ramakrishnan and S. Wong \cite{ramak-wang} proved a cuspidality criterion for the transfer from 
${\rm GL}(2)\times{\rm GL}(3)$ to ${\rm GL}(6)$ and used it to construct new cuspidal cohomology 
classes for ${\rm GL}(6)$. We refer to \cite{raghuram-shahidi-aims} for a brief 
survey of cohomological applications of Langlands functoriality.
H. Kim and F. Shahidi \cite{kim-shahidi-duke} proved a cuspidality criterion for the 
symmetric fourth transfer from ${\rm GL}(2)$ to ${\rm GL}(5)$, which has been used in 
the study of special values of symmetric power $L$-functions by the second author and 
F. Shahidi \cite{raghuram-shahidi-periodrelns}.  Such a potential arithmetic application 
was indeed our original motivation to seek a cuspidality criterion for the exterior square transfer.

\bigskip

{\Small
\noindent{\it Acknowledgements:} It is a pleasure to thank H. Kim for his encouragement of this project. 
We also thank A.~Kable, J.-P.~Labesse, and R.~Zierau for helpful discussions. 
This work is partially supported by an Alexander von Humboldt Research Fellowship for the first author, and 
by an ASR +1 grant from the College of Arts and Science of the Oklahoma State University for the second author.}

\section{Some preliminaries}
\label{sec:prelims}

In this section we collect some results we repeatedly use in later sections. To begin, we recall a 
theorem due to Jacquet and Shalika concerning strong multiplicity one for isobaric automorphic representations.
Then we recall an analytic criterion in terms of Rankin--Selberg $L$-functions, due to Jacquet, Piatetski-Shapiro
and Shalika, that characterizes when two cuspidal automorphic representations are equivalent. 
Next, we note that the natural transfer of automorphic representations of a quasi-split non-split 
general spin group $\GSpin^*(4)$ to $\GL(4)$ is in fact the Asai transfer. Finally, we recall some details 
from the Langlands--Shahidi machinery that will be of use to us, particularly when the ambient group is 
$\GSpin(m)$ with $m=8,10$ or $12$.

\subsection{Some results of Jacquet, Piatetski-Shapiro, and Shalika}
\label{sec:j-ps-s}
The following strong multiplicity one theorem for isobaric representations is 
due to Jacquet and Shalika \cite{jacquet-shalika-1, jacquet-shalika-2}.

\begin{theorem}
\label{thm:isobaric}
Let $\pi_1$ and $\pi_2$ be isobaric automorphic representations of ${\rm GL}(n,{\mathbb A}_F)$. 
Let $S$ be a finite set of places of $F$, containing the archimedean places, such that both $\pi_1$ 
and $\pi_2$ are unramified outside $S$. If $\pi_{1,v} \cong \pi_{2,v}$ for all $v \notin S$, 
then $\pi_1 \cong \pi_2$.
\end{theorem}

Another useful technical tool for us is the following theorem, due to 
Jacquet, Piatetski-Shapiro and Shalika \cite{jacquet-ps-shalika}, concerning Rankin--Selberg $L$-functions. 

\begin{theorem}
\label{thm:rankin-selberg}
Let $\pi_1$ and $\pi_2$ be cuspidal automorphic representations of ${\rm GL}(n_1,{\mathbb A}_F)$
and ${\rm GL}(n_2,{\mathbb A}_F)$, respectively. Let $S$ be a finite set of places containing the archimedean 
places of $F$ and the ramified places of $\pi_1$ and $\pi_2$. The 
partial Rankin--Selberg $L$-function $L^S(s,\pi_1\times\pi_2)$ is holomorphic at $s=1$ unless $n_1=n_2$ and 
$\pi_2 \cong \pi_1^{\vee}$, and in which case it has a simple pole at $s=1$.
\end{theorem}

\subsection{The Asai transfer and the quasi-split non-split $\GSpin^*(4)$}
\label{sec:asai}

Let $E/F$ be a quadratic extension of number fields and let $\Gamma=\Gamma_F$ denote 
the absolute Galois group of $F$. In this section we let $G$ 
denote the group $\GSpin^*(4)$, a quasi-split non-split linear algebraic group over $F$, which 
is isomorphic to the split ${\rm GSpin}(4)$ over $E$. The 
$L$-group of $G$ can be written as $^LG = {\rm GSO}(4,\C) \rtimes \Gamma$, where the Galois 
action, which factors through ${\rm Gal}(E/F)$, is described below. We note that ${\rm GSO}(4,\C)$ 
denotes the special orthogonal similitude group; one can identify it as a quotient of 
$\GL(2,\C)\times\GL(2,\C)$ given by 
\[ {\rm GSO}(4,\C) = \beta (\GL(2,\C) \times \GL(2,\C)), \] 
where $\beta$ is the map on the right of the exact sequence 
\[ 
1 \longrightarrow \C^* \longrightarrow
\GL(2,\C)\times\GL(2,\C) \stackrel{\beta}{\longrightarrow} {\rm GO}(4,\C). 
\] 
For details see \cite[\S 2]{ramak-imrn}. Furthermore, the $\Gamma$-action on 
${\rm GSO}(4,\C)$ is as follows. If $\gamma\in\Gamma$ and $g=\beta(g_1,g_2)$ with 
$g_i\in\GL(2,\C)$, then 
\[ \gamma\cdot g = \begin{cases}
\beta(g_1,g_2) & \mbox{if } \gamma\vert_E = 1, \\
\beta(g_2,g_1) & \mbox{if } \gamma\vert_E \not= 1.
\end{cases} \]

We also need to recall the Asai transfer. Consider 
the group $H=\mbox{Res}_{E/F} \GL(2)$ as a group over $F$. Its $L$-group is given by 
\[ {}^LH = \left(\GL(2,\C)\times\GL(2,\C)\right) \rtimes\Gamma, \] 
with the Galois action given by 
\[ \gamma\cdot (g_1,g_2) = \begin{cases}
(g_1,g_2) & \mbox{if } \gamma\vert_E = 1, \\
(g_2,g_1) & \mbox{if } \gamma\vert_E \not= 1.
\end{cases} \]

Let $W$ be a $2$-dimensional ${\mathbb C}$-vector space, 
and let $V = W \otimes W$. 
After fixing a basis for $W$, we identify ${\GL}(W)$ with ${\rm GL}(2,{\mathbb C})$.
Consider the map 
\[ 
\As : \left( \GL(W) \times \GL(W) \right) \rtimes \Gamma 
\longrightarrow \GL(V) \cong \GL(4,\C) 
\] 
given, on pure tensors, by 
\[
\As(g_1,g_2;\gamma) (\xi_1\otimes\xi_2)= \begin{cases}
g_1 \xi_1 \otimes g_2 \xi_2 & \mbox{if } \gamma\vert_E = 1, \\
g_1 \xi_2 \otimes g_2 \xi_1 & \mbox{if } \gamma\vert_E \not= 1,
\end{cases} 
\]
for all $\xi_i \in W$ and all $g_i \in \GL(W)$. It is straightforward to check that this 
map is indeed a homomorphism. It is called the Asai (or `twisted tensor') 
homomorphism. (Alternatively, one could take the map satisfying 
$\As(g_1,g_2;\gamma) (\xi_1\otimes\xi_2) = - g_1 \xi_2 \otimes g_2 \xi_1$ when 
$\gamma\vert_E \not= 1$. This choice would lead to a quadratic twist of 
the above map. ) 

Further, let $\iota : {}^LG \longrightarrow \GL(V) \cong \GL(4,\C)$ be the map 
defined via 
\[  
\iota(\beta(g_1,g_2);\gamma) (\xi_1\otimes\xi_2)= \begin{cases}
g_1 \xi_1 \otimes g_2 \xi_2 & \mbox{if } \gamma\vert_E = 1, \\
g_1 \xi_2 \otimes g_2 \xi_1 & \mbox{if } \gamma\vert_E \not= 1. 
\end{cases} 
\]
Again it is straightforward to check that this map is indeed an $L$-homomorphism. 
It is now clear that  $\iota\circ(\beta,\mbox{id}) = \As$. In other words,  
the following diagram commutes: 

\[ \xymatrix{
\GL(2,\C)\times\GL(2,\C)\rtimes\Gamma  \ar[rd]_{\As}
\ar@{-}[r]  &^{(\beta,\mbox{id})}\ar[r] & {\rm GSO}(4,\C) \rtimes \Gamma
\ar[ld]^\iota  \\
&\GL(4,\C) &
} \]

Assume that $\pi$ is a cuspidal automorphic representation of $\GL(2,\Ad_E)$ and 
let $\Pi=\As(\pi)$ be its Asai transfer to $\GL(4,\Ad_F)$. (See Krishnamurthy \cite{muthu-imrn} or
Ramakrishnan \cite{ramak-imrn}.) 
Then $\Pi = \iota (\beta(\pi))$, 
where $\beta(\pi)$ denotes the transfer of $\pi$ to the group $\GSpin^*(4,\Ad_F)$, 
and $\iota(\beta(\pi))$ denotes the transfer of $\beta(\pi)$ from $\GSpin^*(4,\Ad_F)$ to 
$\GL(4,\Ad_F)$. The transfer corresponding to $\beta$ exists for formal reasons and the existence of 
the transfer corresponding to $\iota$ (for generic representations)
is part of a joint work of the first author with F. Shahidi \cite{asgari-shahidi-duke,asgari-shahidi-preprint}.

\subsection{The Langlands--Shahidi $L$-functions} 
\label{sec:ls-theory}
Let $P = M N$ be a maximal proper parabolic subgroup of a connected reductive quasi-split 
linear algebraic group $G$, 
where $M$ denotes a Levi subgroup and $N$ denotes the unipotent radical of $P$. 
Let $\sigma$ be a {\it generic} automorphic representation of $M(\Ad_F)$. 
Let $r$ denote the adjoint action of the complex Langlands 
dual group $\widehat{M}$ on the Lie algebra of the dual of $N$. Write $r=r_1 \oplus \cdots \oplus r_m$, 
where the $r_i$'s denote the irreducible constituents of $r$ and the ordering is according to the eigenvalue 
of the adjoint action as in, for example, \cite[p.278]{shahidi-90annals}. The Langlands--Shahidi method 
then constructs the $L$-functions $L(s,\sigma,r_i)$ for $1\le i \le m$. 

We need the following cases of the Langlands--Shahidi method. Let $G=\GSpin(2n+6)$ with 
$n=1,2,3$, and consider a maximal parabolic subgroup of $G$ 
with Levi subgroup $M=\GL(n)\times\GSpin(6)$. 
(One could also work with split spin groups as in \cite[\S 3]{kim-jams}; however, we find it 
more convenient to work with the similitude version of the groups.) 
The algebraic group $\GSpin(6)$ is isomorphic to a quotient of $\GL(1)\times {\rm Spin}(6)$ 
by a central subgroup $A=\{1,(-1,c) \}$, where $c$ is the nontrivial element in the center of 
${\rm Spin}(6)$ of order $2$; see \cite[Proposition 2.2]{asgari-shahidi-duke}. 
The algebraic group $\GL(4)$ is isomorphic 
to a quotient of $\GL(1)\times\SL(4)$ by a cyclic central subgroup $B$ of order 4. 
We identify ${\rm Spin}(6)$ with $\SL(4)$ such that $B$ contains $A$. 
This way we get a natural map, defined over $F$, from 
$\GSpin(6)$ to $\GL(4)$, which in turn induces a map 
\begin{equation}
\label{GSpin4-GL4}
f : M \longrightarrow \GL(n)\times\GL(4). 
\end{equation}

Let $\Pi$ be an irreducible cuspidal automorphic representation of $\GL(4,\Ad_F)$ and let $\sigma$ 
be an irreducible cuspidal automorphic representation of $\GL(n,\Ad_F)$, $n=1,2,3$. 
Choose any irreducible constituent $\Sigma$ of $\sigma\otimes\Pi\vert_{f(M(\Ad_F))}$ and let $\Sigma$ also
denote the corresponding representation of $M(\Ad_F)$. The Langlands--Shahidi method then gives 
\begin{equation}
 L(s,\Sigma, r_1) = L(s, \sigma\otimes\Pi, \rho_n \otimes\wedge^2\rho_4),
\end{equation}   
where $\rho_k$ denotes the standard representation of $\GL(k,\C)$, and the $L$-function on the 
right-hand side is a Langlands $L$-function; see \cite[\S 3]{kim-jams}. 
We record a general fact that we need from the Langlands--Shahidi method. 

\begin{proposition}
\label{LSfact}
Let $w_G$ and $w_M$ denote the longest elements of the Weyl group of $G$ and $M$, respectively. 
Let $w_0=w_G w_M$. If $w_0(\Sigma)\not\cong\Sigma$, then $L(s,\Sigma,r_1)$ is entire. 
\end{proposition}
\begin{proof}
This is a standard fact in the Langlands--Shahidi method. For example, see the proof of 
\cite[Proposition 3.4]{kim-jams}. 
\end{proof}

In order to apply the above proposition one needs to know the action of $w_0$ on a representation 
of $M(\Ad_F)$.

\begin{proposition} Let $G=\GSpin(2n+6)$ with $n$ a positive integer and let $\sigma\otimes\Pi$ 
be a representation of $M(\Ad_F)$ as above. Moreover, let $w_0$ be as above and denote its image 
under the map (\ref{GSpin4-GL4}) by $w_0$ again. Then we have 
\label{w0-action} 
\[
w_0(\sigma\otimes\Pi) = \begin{cases}
\w{\sigma}\otimes(\w{\Pi}\otimes\omega_\sigma) & \mbox{ if $n$ is odd,} \\
\w{\sigma}\otimes(\Pi\otimes\omega_\sigma) & \mbox{ if $n$ is even.} \\
\end{cases}
\] 
Here, $\omega_\sigma$ denotes the central character of $\sigma$. 
\end{proposition}

\begin{proof}
Recall that the nontrivial automorphism of the Dynkin diagram of type $A_m$ corresponds to 
an outer automorphism of ${\rm GL}(m+1)$ and it conjugates an irreducible representation to its dual representation. The proof of the proposition will follow from a description of 
how $w_0$ acts on the root system of type $D_r$.  

We use the Bourbaki notation for the simple roots: 
\[ \alpha_1=\epsilon_1 - \epsilon_2,\cdots, \alpha_{r-2}=\epsilon_{r-2} - \epsilon_{r-1}, 
\alpha_{r-1}=\epsilon_{r-1} - \epsilon_{r},\alpha_{r}=\epsilon_{r-1} + \epsilon_{r}. \] 
The Weyl group of $G$ is isomorphic with $\{\pm 1\}^{r-1} \rtimes \mathcal S_r$ which we 
identify with a subgroup of signed $r\times r$ permutation matrices acting on 
$\rl^r=\rl\epsilon_1\oplus\cdots\oplus\rl\epsilon_r$. With this identification observe that 
\[ w_G = \begin{cases} 
-I_r & \mbox{if $r$ is even,} \\ 
\\
\left( \begin{matrix} -I_{r-1}&\\ & 1\end{matrix} \right) & \mbox{if $r$ is odd.} 
\end{cases} \]

Let $r=n+3$ and let $M$ be a maximal Levi subgroup of type $A_{n-1}\times A_3$ in $G = {\rm GSpin}(2r)$. 
The simple roots in the $A_3$-factor are $\alpha_{r-1},\alpha_{r-2},\alpha_r$. 
The proposition follows by observing that 
\[ w_0(\alpha_{r-1}) = w_G w_M(\alpha_{r-1}) = w_G(-\alpha_{r}) = w_G(-\epsilon_{r-1}-\epsilon_r) 
=\begin{cases} 
\epsilon_{r-1} + \epsilon_{r} = \alpha_{r} & \mbox{if $r$ is even,} \\ 
\epsilon_{r-1} - \epsilon_{r} = \alpha_{r-1} & \mbox{if $r$ is odd,} 
\end{cases} \]
and  
\[ w_0(\alpha_{r}) = w_G w_M(\alpha_{r}) = w_G(-\alpha_{r-1}) = w_G(-\epsilon_{r-1}+\epsilon_r) 
=\begin{cases} 
\epsilon_{r-1} - \epsilon_{r} = \alpha_{r-1} & \mbox{if $r$ is even,} \\ 
\epsilon_{r-1} + \epsilon_{r} = \alpha_{r} & \mbox{if $r$ is odd,} 
\end{cases} \]
while 
\[ w_0(\alpha_{r-2}) = w_G w_M(\alpha_{r-2}) = w_G(-\alpha_{r-2}) = w_G(-\epsilon_{r-2}-\epsilon_{r-1}) 
=\alpha_{r-2} \] 
in either case. Moreover, for $1 \le j \le n-1$ we have 
\[ w_0(\alpha_{j}) = w_G w_M(\alpha_{j}) = w_G(-\alpha_{n-j}) = w_G(-\epsilon_{n-j}+\epsilon_{n-j+1}) 
=\alpha_{n-j}. \] 
This means that $w_0$ induces the nontrivial automorphism of the Dynkin diagram of the $A_{n-1}$-factor of $M$,
and on the $A_3$-factor it induces the nontrivial automorphism of the Dynkin diagram if and only if $r$ is even.

Let $m=m(g,h)$ be an arbitrary element in the Levi subgroup $M$ identified with $\GL(n)\times\GSpin(6)$, 
in $G=\GSpin(2n+6)$, where $g\in\GL(n)$ and $h\in\GSpin(6)$, and let $\nu=\nu(m)$ denote its similitude 
character value. Then 
\[ w_0 m(g,h) w_0^{-1} = m( {^tg}^{-1} \nu(m), h^*), \] 
where 
\[ h^*=\begin{cases}
{^th}^{-1} & \mbox{if $r$ is even,} \\ 
h & \mbox{if $r$ is odd.} \end{cases} 
\] 
We conclude that 
\begin{eqnarray*} 
w_0 (\sigma\otimes\Pi) (m(g,h)) &=& (\sigma\otimes\Pi)(m({^tg}^{-1} \nu(m),  h^*)) \\ 
&=& \sigma(^tg^{-1}) \Pi(h^*) \omega_\sigma(\nu(m)) \\
&=& \w{\sigma}(g) \Pi^*(h) \omega_\sigma(\nu(h)) \\
&=& \left(\w{\sigma}\otimes(\Pi^*\otimes\omega_\sigma)\right) (m(g,h)), 
\end{eqnarray*} 
where 
\[ \Pi^*= \begin{cases}
\w{\Pi} & \mbox{if $r$ is even,} \\ 
\Pi & \mbox{if $r$ is odd.} \end{cases} 
\] 
This completes the proof. Note that $r=n+3$ is even if and only if $n$ is odd.
\end{proof}

\section{The proof of {\bf(ii)}$\Rightarrow${\bf(i)}}
\label{sec:2-implies-1}

We verify that for each of {\bf (ii)(a)} through {\bf (ii)(d)} the exterior square 
transfer $\wedge^2 \Pi$ is not cuspidal. Indeed, it is not difficult to write 
down the isobaric decomposition for $\wedge^2 \Pi$ in each case.

\subsection{{\bf(ii)}(a)$\Rightarrow${\bf(i)}}

\begin{proposition}
\label{prop:gl2xgl2}
Let $\pi_1$ and $\pi_2$ be cuspidal automorphic representations of ${\rm GL}(2,\Ad_F)$.
Let $\pi_1\boxtimes \pi_2$ be the transfer to an automorphic representation of 
${\rm GL}(4,\Ad_F)$, whose existence was established in \cite{ramak-annals}. For brevity, we let 
$\Pi = \pi_1 \boxtimes \pi_2$ and $\omega = \omega_{\pi_1}\omega_{\pi_2}$. 
We have
\begin{itemize}
\item[(a)]
$\wedge^2(\pi_1 \boxtimes \pi_2) = 
\left( {\rm Sym}^2(\pi_1) \otimes \omega_{\pi_2} \right) \boxplus 
\left( {\rm Sym}^2(\pi_2) \otimes \omega_{\pi_1} \right).$
\item[(b)] Assuming Langlands functoriality one should expect
$${\rm Sym}^2(\pi_1 \boxtimes \pi_2) = 
\left( {\rm Sym}^2(\pi_1) \boxtimes {\rm Sym}^2(\pi_2) \right) \boxplus 
\omega_{\pi_1} \omega_{\pi_2}.$$
\item[(c)] The partial L-function $L^S(s,\wedge^2(\Pi)\otimes \omega^{-1})$ is entire while 
the partial L-function $L^S(s, \Pi, {\rm Sym}^2 \otimes \omega^{-1})$ has a pole at $s=1$.
\end{itemize}
\end{proposition}

\begin{proof}
The proof of (a) and (b), using Theorem~\ref{thm:isobaric}, is an easy calculation 
using Satake parameters on both sides. 
More precisely, for a finite place $v$ at which both $\pi_1$ and $\pi_2$ are unramified we let 
$\pi_{1,v}$ and $\pi_{2,v}$ have Frobenius-Hecke eigenvalues 
$t_1={\rm diag}(a_1,b_1)$ and $t_2={\rm diag}(a_2,b_2)$, respectively. Then  
\[ 
\wedge^2 (t_1 \otimes t_2) = \left( {\rm diag}(a_1^2,a_1b_1,b_1^2) \cdot a_2b_2 \right) \boxplus 
\left({\rm diag}(a_2^2,a_2b_2,b_2^2) \cdot a_1b_1\right)
\]
and 
\[
{\rm Sym^2}(t_1 \otimes t_2) = \left( {\rm diag}(a_1^2,a_1b_1,b_1^2) \otimes 
{\rm diag}(a_2^2,a_2b_2,b_2^2) \right) \boxplus \left( a_1b_1 \cdot a_2b_2 \right). 
\]
Part (a) has also been observed by others; see \cite[(7.27)]{muthu-imrn} and 
\cite[(2.6)]{ramak-shahidi}, for example. For (b) to make sense one has to assume 
the symmetric square transfer from ${\rm GL}(4)$ to ${\rm GL}(10)$ and the automorphic tensor product 
from ${\rm GL}(3) \times {\rm GL}(3)$ to ${\rm GL}(9)$, both particular instances of functoriality.

To prove (c), observe that $\Pi^{\vee} \cong \Pi \otimes \omega^{-1}$, which implies 
\[
L^S(s, \Pi \times \Pi^{\vee}) = L^S(s,\wedge^2(\Pi)\otimes \omega^{-1})
L^S(s, \Pi, {\rm Sym}^2 \otimes \omega^{-1}), 
\]
where $S$ is a finite set of places including all the archimedean ones such 
that $\Pi$ is unramified outside of $S$. 
From (a) we have $\wedge^2(\Pi)\otimes \omega^{-1} = {\rm Ad}(\pi_1) \boxplus {\rm Ad}(\pi_2)$. 
(Here ${\rm Ad}(\pi_i) = {\rm Sym}^2(\pi_i)\otimes \omega_{\pi_i}^{-1}$.) If $\pi_i$ is not dihedral, 
then ${\rm Ad}(\pi_i)$ is cuspidal (by Gelbart-Jacquet \cite{gelbart-jacquet})
and hence its partial $L$-function is entire. If $\pi_i$ is dihedral, 
say $\pi_i = I_E^F(\chi)$, then it is easy to see that 
${\rm Ad}(\pi_i) = \omega_{E/F} \boxplus I_E^F(\chi'\chi^{-1})$, where $\omega_{E/F}$ is the 
quadratic character of $F$ associated to $E$ by class field theory, and $\chi'$ is the nontrivial 
${\rm Gal}(E/F)$-conjugate of $\chi$. Since $\pi_i$ is cuspidal, the inducing character $\chi$ is 
Galois regular, i.e., $\chi' \neq \chi$, or equivalently $\chi'\chi^{-1}$ is a nontrivial character, 
whence $L^S(s,{\rm Ad}(\pi_i)) = L^S(s,\omega_{E/F})L^S(s,I_E^F(\chi'\chi^{-1}))$ is entire. 
(In particular, it does not have a pole at $s=1$.) Therefore 
$L^S(s,\wedge^2(\Pi)\otimes \omega^{-1}) = L^S(s,{\rm Ad}(\pi_1))L^S(s,{\rm Ad}(\pi_2))$ 
does not have a pole at $s=1$. However, $L^S(s, \Pi \times \Pi^{\vee})$ has a pole at $s=1$, 
which implies that $L^S(s, \Pi, {\rm Sym}^2 \otimes \omega^{-1})$ has a pole at $s=1$.

Note that (c), unlike (b), is unconditional and does not depend on assuming unproven 
instances of functoriality. 
\end{proof}

\subsection{{\bf(ii)(b)}$\Rightarrow${\bf(i)}}

\begin{proposition}
\label{prop:asai}
Let $E/F$ be a quadratic extension. Let $\pi$ be a cuspidal automorphic representation of 
$\GL(2,\Ad_E)$ and let $\Pi=\As(\pi)$ be its Asai transfer. Assume that $\Pi$ is a cuspidal 
automorphic representation of $\GL(4,\Ad_F)$. Then $\wedge^2\Pi$ is 
cuspidal if and only if $\pi$ is not dihedral. 
\end{proposition}

\begin{proof}
The proof depends on the following identity
$$
\wedge^2(\As(\pi)) = {\rm I}_F^E ({\rm Sym}^2 \pi \otimes \omega_\pi'), 
$$
where $'$ means the nontrivial ${\rm Gal}(E/F)$-conjugate. (See Krishnamurthy \cite[\S7]{muthu-imrn}.)
To begin, assume that $\pi$ is not dihedral. 
By Ramakrishnan \cite[Theorem 1.4]{ramak-imrn} we know that 
$\As(\pi)$ is cuspidal if and only if $\pi'\not\cong\pi\otimes\mu$ 
for any $\mu$. If $\wedge^2\Pi$ is not cuspidal, then 
$$
\left( {\rm Sym}^2 \pi \otimes\omega_\pi'\right)' \cong {\rm Sym}^2 \pi \otimes\omega_\pi'. 
$$ 
This implies that ${\rm Sym}^2 \pi' \otimes\omega_\pi \cong {\rm Sym}^2 \pi \otimes\omega_\pi'$, i.e., 
${\rm Ad}(\pi)\cong {\rm Ad}(\pi')$. This, in turn, implies that $\pi'\cong\pi\otimes\mu$
(by Ramakrishnan \cite[Theorem 4.1.2]{ramak-annals}), contradicting
the fact that there is no such twist. Hence $\wedge^2\Pi$ is cuspidal.

Next, assume that $\pi$ is dihedral. In this case ${\rm Sym}^2(\pi)$ is not cuspidal and therefore, 
${\rm I}_F^E ({\rm Sym}^2 \pi \otimes \omega_\pi')$ cannot possibly be cuspidal.
\end{proof}

\subsection{{\bf(ii)(c)}$\Rightarrow${\bf(i)}}  

\begin{proposition}
\label{prop:gsp4togl4} 
Let $\Pi$ be a cuspidal automorphic representation of $\GL(4,\Ad_F)$ and 
assume that $\Pi$ is a transfer from a cuspidal (generic) automorphic representation 
$\pi$ of ${\rm GSp}(4,\Ad_F)$. Then 
\[ \wedge^2 \Pi = \tilde{r}_5(\pi) \boxplus \omega_\pi, \]
where $\tilde{r}_5$ is a degree $5$ representation of $\GSp(4,\C)$ defined below. 
In particular, $\wedge^2 \Pi$ is not cuspidal. 
\end{proposition}

\begin{proof}
We use Kim \cite[p. 2793]{kim-tams}. As observed there, one has 
\[ \SSp(4,\C) \overset{\iota}{\hookrightarrow} \GL(4,\C) \overset{\wedge^2}{\longrightarrow} 
\GL(6,\C) \] 
and $\wedge^2\circ \iota = r_5 \oplus 1\!\!1$ decomposes into a direct sum of the trivial representation 
and a five-dimensional representation $r_5$. Similarly, 
\[ \GSp(4,\C) \overset{\tilde{\iota}}{\hookrightarrow} \GL(4,\C) \overset{\wedge^2}{\longrightarrow} 
\GL(6,\C) \] 
and $\wedge^2\circ \tilde{\iota} = \tilde{r}_5 \oplus \nu$, where $\nu$ is the similitude 
character of $\GSp(4,\C)$ and $\tilde{r}_5$ is a five-dimensional representation of $\GSp(4,\C)$. 
This implies the desired equality of automorphic representations. 
\end{proof}

\begin{rem} Embedded in the above proof is the assertion that a cuspidal (generic) 
automorphic representation $\pi$ of $\GSp(4,\Ad_F)$ admits a transfer to an automorphic representation $\tilde{r}_5(\pi)$ of 
$\GL_5(\Ad_F)$ corresponding to the representation $\tilde{r}_5$. This depends on
the generic transfer from ${\rm GSp}(4)$ to ${\rm GL}(4)$ (see Asgari-Shahidi \cite{asgari-shahidi-comp}),
and the exterior square transfer from ${\rm GL}(4)$ to ${\rm GL}(6)$ due to Kim \cite{kim-jams}.
\end{rem}

\subsection{{\bf(ii)}(d)$\Rightarrow${\bf(i)}}

\begin{proposition}
\label{prop:gl2-quadratic}
Let $\pi$ be a cuspidal automorphic representation of ${\rm GL}(2,\Ad_E)$, where 
$E/F$ is a quadratic extension, and let $\Pi = {\rm I}_F^E(\pi)$ be the 
automorphic induction of $\pi$ to an automorphic representation of ${\rm GL}(4,\Ad_F)$. 
Then $\wedge^2 \Pi$ is not cuspidal. 
\end{proposition}

\begin{proof}
It is known that
$$
\wedge^2({\rm I}_F^E(\pi))\ = \  
{\rm As}(\pi)\otimes\omega_{E/F}\  \boxplus \ {\rm I}_E^F(\omega_{\pi}),
$$
where $\omega_{E/F}$ is the quadratic Hecke character of $F$ associated to $E/F$ by class field theory. 
See, for example, Kim \cite[\S3]{kim-crm}.
\end{proof}

\section{The proof of {\bf(i)}$\Rightarrow${\bf(iii)}}
\label{sec:1-implies-3}

It is equivalent to prove that if $\Pi$ is a cuspidal automorphic representation 
of $\GL(4,\Ad_F)$ which neither has a nontrivial self-twist nor is essentially self-dual,
then $\wedge^2(\Pi)$ is cuspidal. 
Observe that if an isobaric automorphic representation $\rho$ 
of $\GL(6,\Ad_F)$ is not cuspidal, then it must have an isobaric summand of  
degree $1,2,$ or $3$, i.e., there exists a cuspidal 
representation $\sigma$ of $\GL(n,\Ad_F)$, with $1 \leq n \leq 3$, 
such that $L^S(s, \rho \times \sigma)$ has a pole
at $s=1$. Again $S$ denotes a finite set of places of $F$, 
including all the archimedean ones, such that all the representations involved are unramified at places
outside $S$. 
Now with $\Pi$ as above, the  cuspidality of $\wedge^2\Pi$ follows 
from the following two propositions.

\begin{proposition} 
\label{prop:gl1-gl3-twist}
If $\Pi\not\cong\w{\Pi}\otimes\chi$ for all $\chi$, then 
$L^S(s,\Pi\otimes\sigma,\wedge^2\otimes\rho_2)$ is holomorphic at $s=1$ for 
every cuspidal representation $\sigma$ of $\GL(n,\Ad_F)$, $n=1,3$. 
\end{proposition} 

\begin{proof} 
Let $\Sigma$ be as in (\ref{GSpin4-GL4}). By Proposition~\ref{LSfact}, 
it is enough to show that $w_0(\Sigma)\not\cong\Sigma$. If we have 
$w_0(\Sigma)\cong\Sigma$, then $w_0(\sigma\otimes\Pi)\cong\sigma\otimes\Pi$. 
On the other hand, by Proposition~\ref{w0-action} we have 
$w_0(\sigma\otimes\Pi)\cong\w{\sigma}\otimes(\w{\Pi}\otimes\omega_\sigma)$. 
In particular, we must have $\Pi\cong\w{\Pi}\otimes\omega_\sigma$ contradicting 
the hypothesis. 
\end{proof}

\begin{proposition} 
\label{prop:gl2-twist}
If $\Pi\not\cong\Pi\otimes\chi$ for all nontrivial $\chi$, then 
$L^S(s,\Pi\otimes\sigma,\wedge^2\otimes\rho_2)$ is holomorphic at $s=1$ for 
every cuspidal representation $\sigma$ of $\GL(2,\Ad_F)$. 
\end{proposition} 

\begin{proof} 
The same argument as in the above proof works as long as $\omega_\sigma \not= 1\!\!1$ 
because by Proposition~\ref{w0-action} we have 
$w_0(\sigma\otimes\Pi) \cong \w{\sigma}\otimes(\Pi\otimes\omega_\sigma) 
\not\cong \sigma\otimes\Pi$. This means that if $\sigma$ is a cuspidal representation of 
$\GL(2,\Ad_F)$ with nontrivial central character, 
then $\sigma$ cannot occur as an isobaric summand of $\wedge^2(\Pi)$.

Now suppose that $\sigma$ is a cuspidal representation of 
$\GL(2,\Ad_F)$ with trivial central character and that $\sigma$ is an isobaric summand of $\wedge^2(\Pi)$.
Then the representation $\sigma\otimes\theta^2$ 
occurs in $\wedge^2(\Pi\otimes\theta)$ for any Hecke character $\theta$. 
Note that $\Pi\otimes\theta$ also satisfies the hypothesis that
it has no nontrivial self-twists. 
Choose $\theta$ such that 
\[ \omega_{\sigma\otimes\theta^2} = \omega_\sigma \theta^4 = \theta^4 \not= 1\!\!1 \] 
to get a contradiction. 
\end{proof}

\section{The proof of {\bf(iii)}$\Rightarrow${\bf(ii)}}
\label{sec:3-implies-2}

We prove that if $\Pi$ satisfies {\bf(iii)}($\beta$), then it  is of the form {\bf(ii)}(d),
and if it satisfies {\bf(iii)}($\alpha$), then it is one of {\bf(ii)}(a)--(c).

\subsection{{\bf(iii)}($\beta$) $\Longrightarrow$ {\bf(ii)}(d)} 

Assume that 
\begin{equation}\label{chi4}
\Pi\cong\Pi\otimes\chi
\end{equation} 
for some nontrivial $\chi$. Taking central 
characters we have $\chi^4=1\!\!1$. If $\chi^2\not=1\!\!1$, then we may replace 
$\chi$ with $\chi^2$ in (\ref{chi4}), which means we may assume that the character $\chi$ 
in (\ref{chi4}) is quadratic. We want to show that $\Pi$ is induced from a quadratic extension. 
If $\Pi$ is a representation of $\GL(2)$, then the analogous statement is a 
well-known result due to Labesse-Langlands \cite{labesse-langlands}. In our case it follows from 
the work of Arthur-Clozel \cite{arthur-clozel} and some $L$-function arguments as we 
explain below. 

\begin{lemma}
Let $\Pi$ be a cuspidal representation of $\GL(2n,\Ad_F)$ satisfying $\Pi\cong\Pi\otimes\chi$ 
for a nontrivial quadratic character $\chi$. Then $\Pi={\rm I}_E^F(\pi)$, where $E/F$ 
is the quadratic extension associated with $\chi$ and $\pi$ is a cuspidal representation of 
$\GL(n,\Ad_E)$. 
\end{lemma}

\begin{proof}
We first claim that the base change $\Pi_E$ is not cuspidal. 
To see this, assume that it is cuspidal. For a finite 
set $S$ of places of $F$ and the corresponding set $T$ of places of $E$ lying above those 
in $S$, as before, we have 
\begin{eqnarray}\label{DoublePole}
L^T(s,\Pi_E\times\w{\Pi}_E) &=& L^S(s,\Pi\times\w{\Pi}) L^S(s,\Pi\times\w{\Pi}\otimes\chi) \\
\nonumber &=& L^S(s,\Pi\times\w{\Pi})^2.
\end{eqnarray}
For sufficiently large $S$, the left hand side of (\ref{DoublePole}) has a simple pole at 
$s=1$ while the right hand side has a double pole at $s=1$. This contradiction shows that 
$\Pi_E$ is not cuspidal. This means that $\Pi_E=\pi_1\boxplus\pi_2$, where $\pi_i$ are 
cuspidal representations of $\GL(n,\Ad_E)$. Further, $\pi_1 \not\cong \pi_2$, because 
if they are equivalent, then 
$$
L^S(s,\Pi\times\w{\Pi})^2 = 
L^T(s,\Pi_E\times\w{\Pi}_E) = L^T(s,\pi_1\times\w{\pi}_1)^4,
$$
but $L^S(s,\Pi\times\w{\Pi})^2$ has a double pole and  
$L^T(s,\pi_1\times\w{\pi}_1)^4$ has a pole of order $4$ at $s=1$. 

Next, we claim that $\Pi\cong{\rm I}_E^F(\pi_1)$. 
To show this it is enough to prove that the partial $L$-function 
$L^S(s,{\rm I}_E^F(\pi_1)\times\w{\Pi})$ has a simple pole at $s=1$. This follows from 
\begin{eqnarray*}
L^S(s,{\rm I}_E^F(\pi_1)\times\w{\Pi}) 
& = & L^S(s,{\rm I}_E^F(\pi_1\times\w{\Pi}_E)) \\
& = & L^T(s,\pi_1\times\w{\Pi}_E) \\
& = & L^T(s,\pi_1\times\w{\pi}_1)L^T(s,\pi_1\times\w{\pi}_2).
\end{eqnarray*}
Since $\pi_2 \not\cong\pi_1$ we know that $L^T(s,\pi_1\times\w{\pi}_1)L^T(s,\pi_1\times\w{\pi}_2)$ 
has a simple pole at $s=1$. 
\end{proof}

\subsection{{\bf(iii)}($\alpha$) $\Longrightarrow$ {\bf(ii)}(a)--(c)}

Now assume that 
\begin{equation}
\label{chi2}
\Pi \cong \w{\Pi} \otimes \chi
\end{equation} 
for some $\chi$. For a finite set $S$ of places of $F$, as before, we have 
\begin{eqnarray*} 
L^S(s,\Pi\times\w{\Pi}) 
&=& L^S(s,\Pi\times(\Pi\otimes\chi^{-1})) \\
&=& L^S(s,(\Pi\times\Pi)\otimes\chi^{-1}) \\
&=& L^S(s,\Pi,\wedge^2\otimes\chi^{-1}) L^S(s,\Pi, \mbox{Sym}^2\otimes\chi^{-1}). 
\end{eqnarray*}
The last two $L$-functions are the standard twisted exterior square and twisted 
symmetric square $L$-functions of $\Pi$.  
If $S$ is sufficiently large, then $L^S(s,\Pi\times\w{\Pi})$ has a simple pole 
at $s=1$. Therefore one and exactly one of the partial twisted exterior or 
symmetric square $L$-functions has a simple pole at $s=1$. 

First, assume that $L^S(s,\Pi,\wedge^2\otimes\chi^{-1})$ has a pole at $s=1$. Then 
there exists a cuspidal representation $\pi$, which may be taken to be globally 
generic, of $\GSp(4,\Ad_F)$ such that $\Pi$ is the functorial transfer of $\Pi$, i.e., 
$\Pi$ is a representation as in {\bf(ii)}(c). This 
result has been known for a long time and, we believe, is originally due to Jacquet, 
Piatetski-Shapiro, and Shalika. See Gan-Takeda \cite{gan-takeda} for a 
proof. It would also follow from the more general method of ``descent'' as we explain 
below.  

Next, assume that $L^S(s,\Pi, \mbox{Sym}^2\otimes\chi^{-1})$ has a pole at $s=1$. 
Taking central characters in (\ref{chi2}) we have $\omega_\Pi=\omega_\Pi^{-1} \chi^4$. 
In other words, $\mu=\omega_\Pi\chi^{-2}$ is a quadratic character. If $\mu$ is trivial, 
then $\Pi$ is a transfer from a cuspidal representation $\pi$ of $\GSpin(4)$, a 
split connected reductive group of type $D_4$ whose derived group is ${\rm Spin}(4)$. 
If $\mu$ is nontrivial, then $\Pi$ is a functorial transfer from the 
quasi-split non-split group $\GSpin^*(4)$ associated with the quadratic extension $E/F$ attached to $\mu$. 

These facts can be proved using the ``descent'' method of Ginzburg-Rallis-Soudry. 
If $\chi$ is trivial, then $\Pi$ would be a transfer from a special orthogonal or 
symplectic group. We refer to Ginzburg-Rallis-Soudry \cite[Theorem A]{grs} and 
Soudry \cite[Theorem 4 and 12]{soud} for the proofs for classical groups. 
For the more general case of classical similitude groups 
we understand that J. Hundley and E. Sayag are now 
in the process of writing down a proof of this fact among other things. For certitude, 
we state the precise statement we are using from the descent theory below. 

\begin{quote}
Let $\Pi$ be a cuspidal automorphic representation of $\GL(4,\Ad_F)$ satisfying 
$\Pi\cong\w{\Pi}\otimes\chi$. Assume that the partial $L$-function 
$L^S(s,\Pi,\mbox{Sym}^2\otimes\chi^{-1})$ has a pole at $s=1$ for a sufficiently large 
finite set $S$ of places of $F$. If $\mu=\omega_\Pi \chi^{-2}$ is trivial, then there 
exists a globally generic cuspidal representation $\pi$ of  
$\GSpin(4,\Ad_F)$ such that $\pi$ transfers to $\Pi$. If $\mu$ is a nontrivial (necessarily 
quadratic) character, and $E/F$ is the associated quadratic extension, then there 
exists a globally generic cuspidal representation $\pi$ of 
$\GSpin^*(4,\Ad_F)$ such that $\pi$ transfers to $\Pi$. (In our notation $\GSpin^*(4)$ 
is a non-split quasi-split group over $F$ which is isomorphic to the split $\GSpin(4)$ over $E$.)
\end{quote}

With the above notation, if $\mu$ is trivial, then $\Pi$ is as in  
{\bf(ii)}(a), and if it is nontrivial then $\Pi$ is as in {\bf(ii)}(b).

\section{Examples and Complements} 
\label{sec:examples}
In this section we give a few examples of our main result. In some of them the proposed 
representation $\Pi$ of $\GL(4,\Ad_F)$ is not yet proved to be automorphic, but it is 
conjecturally so.  We also comment on possible intersection among the four cases in 
part {\bf (ii)} of Theorem~\ref{thm:main}. Finally, 
we present a theorem due to A. Kable on when the exterior square of an irreducible four-dimensional
representation is reducible. 

\subsection{K. Martin's $\bf G_{192}$} 
\label{sec:martin}

The matrices 
{\Small
\[ 
a = \left[   \begin{matrix} 
      -1 &&& \\
      &-1&& \\
      &&1& \\
      &&&1
   \end{matrix} \right], \  
b = \left[   \begin{matrix} 
      -1 &&& \\
      &1&& \\
      &&-1& \\
      &&&1
   \end{matrix} \right],  \ 
c = \left[   \begin{matrix} 
      &&-1& \\
      &&&1 \\
      -i&&& \\
      &-i&&
   \end{matrix} \right],  \ 
d = \left[   \begin{matrix} 
      &-1&& \\
      &&-1& \\
      1&&& \\
      &&&1
   \end{matrix} \right] \]} 

\noindent
in ${\rm GL}(4,{\mathbb C})$ generate a group $G_{192}$ of order $192$.  
Let $\rho$ be the four-dimensional representation of the group
$G_{192}$ given by inclusion. Then $\rho$ is an irreducible representation. 
In \cite{kimball} K.~Martin showed that $\rho$ is modular, i.e., there exists a (cuspidal) automorphic
representation $\Pi(\rho)$ that corresponds to $\rho$. 

\begin{exam}
\label{exam:martin}
Let $\Pi_1 = \Pi(\rho)$. Then $\Pi_1$ is a
cuspidal automorphic representation of $\GL(4,\Ad_F)$. Moreover, it is neither essentially self-dual nor does it 
have a nontrivial self-twist. It is not on the list of possibilities of {\bf (ii)}. Furthermore, 
$\wedge^2\Pi_1$ is cuspidal. In other words, $\Pi_1$ is an example of a cuspidal representation which 
does not satisfy any of {\bf (i)}--{\bf (iii)} of Theorem~\ref{thm:main}.
\end{exam}

\begin{proof}
First we check that $\wedge^2(\rho)$ is an irreducible representation. To see this 
consider the standard basis $\langle e_1,e_2,e_3,e_4\rangle$ for $\C^4$ and fix 
the ordered basis $\langle w_1,w_2,\dots,w_6\rangle$ 
of $\C^6=\wedge^2 \C^4$ given by
\[
\begin{matrix}
w_1=e_1\otimes e_2 - e_2\otimes e_1, \ 
w_2=e_1\otimes e_3 - e_3\otimes e_1, \ 
w_3=e_1\otimes e_4 - e_4\otimes e_1, 
\\
w_4=e_2\otimes e_3 - e_3\otimes e_2, \ 
w_5=e_2\otimes e_4 - e_4\otimes e_2, \ 
w_6=e_3\otimes e_4 - e_4\otimes e_3. 
\end{matrix}
\] 
Let $A,B,C,D$ be the images of $a,b,c,d$ under $\wedge^2$, respectively. Then, with 
respect to the above basis, we have
\[
A = \left[   \begin{matrix} 
      1&&&&& \\
      &-1&&&& \\
      &&-1&&& \\
      &&&-1&& \\
      &&&&-1& \\
      &&&&&1 
   \end{matrix} \right],  
B = \left[   \begin{matrix} 
      -1&&&&& \\
      &1&&&& \\
      &&-1&&& \\
      &&&-1&& \\
      &&&&1& \\
      &&&&&-1 
   \end{matrix} \right], 
\] \[  
C = \left[   \begin{matrix} 
      &&&&&-1 \\
      &-i&&&& \\
      &&&-i&& \\
      &&i&&& \\
      &&&&i& \\
      1&&&&& 
   \end{matrix} \right], 
D = \left[   \begin{matrix} 
      &&&1&& \\
      1&&&&& \\
      &&&&-1& \\
      &1&&&& \\
      &&&&&-1 \\
      &&1&&& 
   \end{matrix} \right].   
\]
It is easy to check that if a $6\times 6$ matrix $X$ commutes with $A,B$, and $C$, then 
it has to be of the form 
\[ X = \left[   \begin{matrix} 
      a&&&&&b \\
      &c&&&& \\
      &&e&f&& \\
      &&-f&e&& \\
      &&&&d& \\
      -b&&&&&a 
   \end{matrix} \right].   
\] 
Further, if $XD=DX$, then $a=c=d=e$ and $b=f=0$, i.e., $X$ is a scalar matrix. 
Therefore, ${\rm Hom}_{G_{192}}(\wedge^2\rho,\wedge^2\rho) = \C$ and, by Schur's lemma, 
the representation $\wedge^2\rho$ is irreducible. This implies that $\wedge^2\Pi_1$
is cuspidal. (This is because if a complex Galois representation $\sigma$ is modular, i.e.,
corresponds to an automorphic representation $\pi = \pi(\sigma)$, then 
$\sigma$ is irreducible if and only if $\pi$ is cuspidal. This fact follows from  
$L(s,\sigma\otimes\sigma^{\vee}) = L(s,\pi\times\pi^{\vee})$; 
see Ramakrishnan \cite[Introduction]{ramak-irr-cusp}.) 
Martin observes that $\rho$, and hence $\Pi_1$, is not essentially self-dual. 
Clearly $\Pi_1$ is not on the list of possibilities in {\bf (ii)} of the main theorem 
because if it were, then $\wedge^2\Pi_1$ would not be cuspidal (see Section \ref{sec:2-implies-1}). 
Hence $\Pi_1$ does not satisfy any of the equivalent statements of Theorem~\ref{thm:main}.
\end{proof}

In \cite{kimball} Martin considers a four-dimensional irreducible representation $\rho$ of
the absolute Galois group of ${\mathbb Q}$ whose image in ${\rm PGL}(4,{\mathbb C})$,
denoted $\bar{G}$, is an extension of $A_4$ by $V_4$. In this situation 
$\bar{G}$ is either $V_4\rtimes A_4$  or $V_4 \cdot A_4$.  
In the former case $\rho$ is of ${\rm GO}(4)$-type and 
$\Pi = \Pi(\rho)$ is a transfer from $\GL(2,\Ad_F) \times \GL(2,\Ad_F)$, which is contained in 
our {\bf(ii)}(a). The example of $G_{192}$ is an instance of the latter situation. In either case,
$\Pi(\rho)$ may also be thought of as being obtained by automorphic induction across a non-normal 
quartic extension with no quadratic subextension.

\subsection{The standard representation of $S_5$}
\label{sec:s5-a5}

Consider a tower of number fields $\tilde{J}/J/E/F$. Here, $E/F$ and
$\tilde{J}/J$ are quadratic extensions, $J/E$ is an $A_5$-extension, $J/F$ is an
$S_5$-extension, $\tilde{J}/E$ is an $\SL(2,{\mathbb F_5})$-extension, and 
$\tilde{J}/F$ is Galois. 
(Recall that $A_5\cong \PSL(2,{\mathbb F_5})$.) 
In what follows we identify $A_5$, $S_5$ and $\SL(2,{\mathbb F_5})$ with
these Galois groups.

Let $\sigma$ be the standard four-dimensional irreducible representation of $S_5$, 
\[ \sigma : S_5 \longrightarrow \GL(4,\C). \] 
Here are some properties of $\sigma$: 
\begin{enumerate}
\item $\wedge^2\sigma$ is a six-dimensional irreducible representation of $S_5$ (see 
\cite[\S3.2]{fulton-harris}). 
\item $\sigma|_{A_5}$ is irreducible (because $\sigma\not\cong\sigma\otimes\epsilon$, 
where $\epsilon$ is the sign character of $S_5$).  
\item $\wedge^2(\sigma|_{A_5}) = (\wedge^2\sigma)|_{A_5}$ is reducible (because 
$\wedge^2(\sigma)\cong\wedge^2(\sigma)\otimes\epsilon$), and its irreducible 
constituents are both of degree $3$. 
\item $\sigma$ is self-dual (because the character of $\sigma$ has integer values). 
\item $\sigma|_{A_5} = \rho_1 \otimes \rho_2$, where $\rho_1$ and $\rho_2$ are the two-dimensional 
irreducible representations of $\SL(2,\mathbb F_5)$ (see \cite[Lemma 5.1]{kim-inv}). 
\end{enumerate}
We need some details about the $\rho_i$. Let $\tilde{\rho}$ be the unique (up to twists) cuspidal
representation of $\GL(2,\mathbb F_5)$ whose restriction to $\SL(2,\mathbb F_5)$ is reducible.
In this case, $\tilde{\rho}|_{\SL(2,\mathbb F_5)} = \rho_1 \oplus \rho_2$. 
If $g \in \GL(2,\mathbb F_5) - 
Z({\mathbb F}_5)\SL(2,\mathbb F_5)$, then $\rho_1^g = \rho_2$. Here $Z$ is the center of 
$\GL(2)$. Conjugating $\SL(2,\mathbb F_5)$ by 
such an element $g$ induces the nontrivial outer automorphism of $\SL(2,\mathbb F_5)$ because 
if it were an inner automorphism, then we would have $\rho_1 \cong \rho_2$, which contradicts the fact that
the restriction from $\GL_2$ to $\SL_2$ is multiplicity free.

Let $\rho$ denote either $\rho_1$ or $\rho_2$. 
In constructing examples (to illustrate our main theorem), we make the following
{\it assumption: $\rho$ is modular}, i.e., there exists 
a cuspidal automorphic representation $\pi(\rho)$ of $\GL(2,\Ad_E)$ with
$\rho\leftrightarrow\pi(\rho)$. In this situation, it is 
expected \cite{kim-private} that there exists an automorphic representation $\pi(\sigma)$ of 
$\GL(4,\Ad_F)$ with $\pi(\sigma)\leftrightarrow\sigma$ and $\pi(\sigma)$  is the Asai transfer of 
$\pi(\rho)$, i.e., $\pi(\sigma) = \As\left(\pi(\rho)\right)$.

\begin{exam}
\label{exam:s5}
Let $\Pi_2 =\pi(\sigma) = \As(\pi(\rho))$. Then $\Pi_2$ is a 
cuspidal automorphic representation of $\GL(4,\Ad_F)$ which is self-dual. However, it is the Asai transfer
of a nondihedral cuspidal representation. Moreover, it has no nontrivial self-twists and 
it is not on the list of possibilities in {\bf (ii)}. Furthermore, its exterior square transfer
$\wedge^2\Pi_2$ is cuspidal. In other words, $\Pi_2$ is an example of a cuspidal representation which 
does not satisfy any of {\bf (i)}--{\bf (iii)} of Theorem~\ref{thm:main}.
\end{exam}

\begin{proof}
Since $\Pi_2 = \pi(\sigma)$ and $\sigma$ is irreducible, we conclude that $\Pi_2$ is cuspidal. 
(Cuspidality of $\Pi_2$ may also be seen by appealing to 
the cuspidality criterion for the Asai transfer due to Ramakrishnan \cite[Theorem 1.4]{ramak-imrn}.)

Next, we note that $\Pi_2$ is self-dual and is the Asai transfer of a cuspidal representation, namely 
$\pi(\rho)$. Note that $\pi(\rho)$ is not dihedral as $\rho$ is not induced from a character of an
index two subgroup because there is no such subgroup in ${\rm SL}(2,{\mathbb F}_5)$.
Also, $\Pi_2$ has no nontrivial self-twists because $\sigma$ has no nontrivial self-twists. 

Finally, note that $\wedge^2 \Pi_2$ is cuspidal since 
$\wedge^2 \Pi_2 = \wedge^2\pi(\sigma) = \pi(\wedge^2\sigma)$ and $\wedge^2\sigma$ 
is an irreducible representation implying that $\pi(\wedge^2\sigma)$ is 
cuspidal. (See, for example, Ramakrishnan \cite[Introduction]{ramak-irr-cusp}.)
\end{proof}

\begin{exam}
\label{exam:a5}
Let $\Pi_3 = (\Pi_2)_E$ be the base change of $\Pi_2$ to an automorphic
representation of ${\rm GL}(4,\Ad_E)$.
Then $\Pi_3$ is a cuspidal automorphic representation of $\GL(4,\Ad_E)$ 
which is self-dual and is not the Asai transfer of a nondihedral representation. 
It is contained in {\bf (ii)}(a) and its exterior square transfer $\wedge^2\Pi_3$ is not cuspidal. 
In other words, $\Pi_3$ is an example of a cuspidal representation which 
satisfies {\bf (i)}--{\bf (iii)} of Theorem~\ref{thm:main}.
\end{exam}

\begin{proof}
To see cuspidality of $\Pi_3$, as well as the fact that it is contained in {\bf (ii)}(a), note that
\[ 
\Pi_3 = (\Pi_2)_E = \pi(\sigma)_E = \pi(\sigma|_{A_5}) = 
\pi(\rho_1\otimes\rho_2) = \pi(\rho_1) \boxtimes \pi(\rho_2). 
\]
Neither $\rho_i$ is monomial since ${\rm SL}(2,{\mathbb F}_5)$ does not have an index two 
subgroup. Applying the cuspidality criterion for $\pi(\rho_1) \boxtimes \pi(\rho_2)$
due to Ramakrishnan \cite[Theorem 11.1]{ramak-tifr}, we see that $\Pi_3$ is not cuspidal if and only
if $\pi(\rho_1) \cong \pi(\rho_2)\otimes\mu$ for some Hecke character $\mu$ of $E$. 
On the other hand, $\pi(\rho_2)\otimes\mu = \pi(\rho_2\otimes \mu)$, where we identify the
Hecke character $\mu$ with a character of the absolute Galois group of $E$
via global class field theory. Hence, we have 
$\pi(\rho_1) = \pi(\rho_2\otimes\mu)$. This implies that
$\rho_1\cong\rho_2\otimes\mu$ (since, for any two Galois
representations $\tau_1$ and $\tau_2$, one has $\pi(\tau_1)\cong\pi(\tau_2)$ if
and only if $\tau_1\cong\tau_2$; one can see this by considering the equality 
$L^S(s, \pi(\tau_1) \times \pi(\tau_2)^{\vee}) = L^S(s, \tau_1 \otimes \tau_2^{\vee})$).
Therefore, $\mu$ is a character of $\mathrm{Gal}(\tilde{J}/E)=\SL(2,{\mathbb F}_5)$, 
a perfect group, hence $\mu$ is trivial. Whence $\rho_1=\rho_2$, which contradicts 
the fact that they are inequivalent as was observed earlier.

Next, observe that $\wedge^2\Pi_3$ is not cuspidal because 
\[
\wedge^2\Pi_3 = \wedge^2(\pi(\rho_1) \boxtimes \pi(\rho_2)) = 
({\rm Sym}^2(\pi(\rho_1))\otimes\omega_{\pi(\rho_2)}) \oplus 
({\rm Sym}^2(\pi(\rho_2))\otimes\omega_{\pi(\rho_1)}),
\]
which is of isobaric type $(3,3)$. (See Proposition~\ref{prop:gl2xgl2}.)

Finally, we observe that $\Pi_3$ is self-dual because $\sigma$,
and hence $\sigma|_{A_5}$, is self-dual and that 
$\Pi_3$ could not be an Asai transfer of a nondihedral representation because if it were, then 
$\wedge^2(\Pi_3)$ would be cuspidal by Proposition~\ref{prop:asai}.
\end{proof}

\subsection{On possible intersections between representations in {\bf (ii)}.}

The purpose of this subsection is to show that the cases {\bf (ii)}(a) through 
{\bf (ii)}(d) are not mutually exclusive.

\begin{exam}
Let $\pi = {\rm I}_E^F(\chi)$ be a cuspidal automorphic representation of 
${\rm GL}(2,{\mathbb A}_F)$ which is automorphically induced from a Hecke character $\chi$ of 
$E$, where $E/F$ is a quadratic extension.
Let $\tau$ be a nondihedral cuspidal automorphic representation of ${\rm GL}(2,{\mathbb A}_F)$.
Let $\Pi_4 = \pi \boxtimes \tau$. 
Then $\Pi_4$ is a representation that is common to {\bf (ii)}(a), 
{\bf (ii)}(b) and {\bf (ii)}(c). 
\end{exam}

\begin{proof}
From Ramakrishnan's cuspidality criterion \cite[Theorem 11.1]{ramak-tifr}
we know that $\Pi_4$ is a cuspidal representation of ${\rm GL}(4,{\mathbb A}_F)$. By construction, 
$\Pi_4$ is in {\bf (ii)}(a). 

We observe that
$$
\Pi_4 = {\rm I}_E^F(\chi) \boxtimes \tau = {\rm I}_E^F(\chi \otimes \tau_E).
$$
Since the induced representation ${\rm I}_E^F(\chi \otimes \tau_E)$ is cuspidal, the 
inducing representation $\chi \otimes \tau_E$ is, {\it a fortiori}, cuspidal. Hence 
$\Pi_4$ is in {\bf (ii)}(b). 

Now we claim that 
$\Pi$ is also a transfer from a (generic) cuspidal representation of ${\rm GSp}(4,{\mathbb A}_F)$.
To see this we recall the following well known identities: 
\begin{eqnarray*}
{\rm Sym}^2({\rm I}_E^F(\chi)) & = & {\rm I}_E^F(\chi^2) \boxplus \chi |_{{\mathbb A}_F^{\times}}, \\
\wedge^2({\rm I}_E^F(\chi)) & = & \chi |_{{\mathbb A}_F^{\times}} \cdot \omega_{E/F}, 
\end{eqnarray*}
where $\omega_{E/F}$ is the quadratic Hecke character of $F$ associated to $E/F$ by class field theory. 
In particular, the central character of $\pi$ is given by $\omega_{\pi} = 
\chi |_{{\mathbb A}_F^{\times}} \cdot \omega_{E/F}$. 
From Proposition \ref{prop:gl2xgl2} we have
\[
\wedge^2(\pi \boxtimes \tau) = 
\left( {\rm Sym}^2(\pi) \otimes \omega_{\tau} \right) \boxplus 
\left( {\rm Sym}^2(\tau) \otimes \omega_{\pi} \right). 
\]
For brevity write $\omega = \omega_{\pi}\omega_{\tau}$. We deduce that 
\[
\wedge^2(\Pi_4) = \left({\rm I}_E^F(\chi^2)\otimes\omega_{\tau}\right)
\boxplus \omega \omega_{E/F}
\boxplus \left( {\rm Sym}^2(\tau) \otimes \omega_{\pi} \right). 
\]
Hence, the partial $L$-function 
$L^S(s, \Pi, \wedge^2\otimes (\omega \omega_{E/F})^{-1})$ has a pole at $s=1$.  
Applying a recent result of Gan and Takeda \cite{gan-takeda} we 
conclude that $\Pi$ is a transfer from ${\rm GSp}(4)$, i.e., $\Pi$ is in {\bf (ii)}(c).
\end{proof}

\subsection{A calculation on the Galois side}
\label{sec:galois}

As mentioned in the introduction, one may ask for an irreducibility criterion on the {\it Galois side}, 
i.e., for $\ell$-adic Galois representations, which reflects the cuspidality criterion one is looking
for. In this section we present such a theorem due to A.~Kable.
We are grateful to him for the permission to include this material here. 
Theorem~\ref{thm:galois-side} below is the analogue of 
the equivalence of {\bf (i)} and {\bf (iii)} in Theorem~\ref{thm:main}. We begin by reviewing some preliminaries.

Let $k$ be an algebraically closed field whose characteristic is not two
and let $V$ be a  four-dimensional $k$-vector space. Fix a nonzero element
$\eta \in \wedge^4V$. There is a nondegenerate symmetric bilinear form
$B$ on $\wedge^2V$ defined by $\omega_1\wedge\omega_2=B(\omega_1,\omega_2)\eta$
for all $\omega_1,\omega_2\in \wedge^2V$. The bilinear space $(\wedge^2V,B)$ is isomorphic
to the orthogonal sum of three hyperbolic planes. Let
$\mathrm{GO}(B)$ be the group of similitudes of $B$,
$\lambda:\mathrm{GO}(B)\to k^{\times}$ the similitude character, and
$\mathrm{GSO}(B)$ the subgroup of proper similitudes. This subgroup
consists of those $T\in\mathrm{GO}(B)$ such that $\det(T)=\lambda(T)^3$, 
and it coincides with the connected component of the identity in the
algebraic group $\mathrm{GO}(B)$. (In the literature, the group ${\rm GSO}(B)$
is also denoted by ${\rm SGO}(B)$ or ${\rm GO}^+(B)$.)
We use similar notation also for
the similitude groups of forms on $V$ itself. Let
$\rho:\GL(V)\to\GL(\wedge^2V)$ be the homomorphism $\rho(S)=\wedge^2S$.

Let $G$ be a group and let $\sigma$ be an irreducible representation of $G$ on
$V$. Recall that $\sigma$ is \emph{essentially self-dual} if there is a
character $\chi$ of $G$ such that $\sigma^{\vee}\cong\chi\otimes\sigma$. In this
case, $\chi^{-1}$ is a subrepresentation of $\sigma\otimes\sigma$. We
say that $\sigma$ has \emph{symplectic type} if $\chi^{-1}$ occurs in
$\wedge^2\sigma$ and \emph{orthogonal type} if $\chi^{-1}$ occurs in
${\rm Sym}^2\sigma$. If $\sigma$ is essentially self-dual of orthogonal type, 
then there is a nonzero symmetric bilinear form $C$ on $V$ such
that $G$ acts on $V$ by similitudes of $C$. The kernel of $C$ is a
$G$-invariant proper subspace of $V$ and hence trivial. Thus, $C$ is
nondegenerate and $\sigma(G)\subset\mathrm{GO}(C)$. If 
$\sigma(G)\subset\mathrm{GSO}(C)$, then we say that $\sigma$ is of
\emph{proper orthogonal type}; otherwise, we say that $\sigma$ is of
\emph{improper orthogonal type}. Finally, we say that $\sigma$ has a
\emph{nontrivial quadratic self-twist} if there is a nontrivial $\{\pm 1\}$-valued
character $\chi$ of $G$ such that $\sigma\cong\chi\otimes\sigma$.

\begin{theorem}[A.~Kable]
\label{thm:galois-side}
Let $\sigma$ be an irreducible $4$-dimensional representation of a
group $G$ over an algebraically closed field whose characteristic
is not two. Then the following two conditions on $\sigma$ are equivalent:
\begin{enumerate}
\item $\wedge^2\sigma$ is reducible.
\item $\sigma$ satisfies at least one of the following:
  \begin{enumerate}
  \item is essentially self-dual of symplectic type, 
  \item has a nontrivial quadratic self-twist, or
  \item is essentially self-dual of proper orthogonal type.
  \end{enumerate}
\end{enumerate}
\end{theorem}

Toward the proof of the above theorem, we begin with a lemma.

\begin{lemma}
\label{lem:omnibus}
The bilinear space $(\wedge^2V,B)$ has the following properties.
\begin{enumerate}
\item[(1)] The image of $\rho$ is $\mathrm{GSO}(B)$.
\item[(2)] An isotropic line in $\wedge^2V$ has the form $\wedge^2Q$, where
$Q<V$ is a uniquely determined $2$-dimensional subspace of $V$.
\item[(3)] An isotropic $3$-space in $\wedge^2V$ either has the form $L\wedge V$,
where $L<V$ is a uniquely determined line, or the form $\wedge^2U$, where
$U<V$ is a uniquely determined $3$-space.
\item[(4)] Let $W < \wedge^2V$ be a $3$-space on which $B$ is nondegenerate.
Then there is a nondegenerate symmetric bilinear form $C$ on $V$
such that
\begin{equation*}
\{g\in\GL(V)\mid \rho(g)(W)=W \}=\mathrm{GSO}(C).
\end{equation*}
The form $C$ is determined by $W$ up to scalars. 
Every nondegenerate symmetric bilinear form on $V$ occurs in this
way for a suitable choice of $W$.
\end{enumerate}
\end{lemma}

\begin{proof}
We omit the proofs of (1), (2), and (3), as they are easy exercises,
and briefly sketch the proof of (4). Let $Q$ be the quadric
hypersurface in $\mathbb{P}(\wedge^2 V)$ consisting of null vectors
for $B$. By (2), we may identify $Q$ with the Grassmannian of lines in
$\mathbb{P}(V)$. Let $W$ be a $3$-dimensional subspace of $\wedge^2V$
on which $B$ is nondegenerate, and $Y=\mathbb{P}(W)\cap Q$ be the
smooth plane conic defined by $B\vert_{W}$. Let $T$ be the subvariety
of $\mathbb{P}(V)$ obtained by taking the union of the lines in
$\mathbb{P}(V)$ corresponding to points of $Y$. It is easily verified
that $T$ is a smooth quadric hypersurface, and so there is a
nondegenerate symmetric bilinear form $C$ on $V$, unique up to
scalars, such that $T$ has equation $C(v,v)=0$. An element
$g\in\mathrm{GL}(V)$ preserves $T$ together with the ruling $T\to Y$
sending a line in $T$ to the corresponding point of $Y$ if and only if
$g\in\mathrm{GSO}(C)$. From the construction, the set of $g$ with this
property is the same as the set of all $g$ such that $\rho(g)(W)=W$. The
last claim in (4) follows from the fact that $\mathrm{GL}(V)$ acts
transitively on the set of all nondegenerate symmetric bilinear forms on
$V$.
\end{proof}

\begin{proposition}
\label{prop:1-subspace}
The representation $(\sigma,V)$ is essentially self-dual of symplectic
type if and only if $\wedge^2V$ contains a $G$-invariant line.
\end{proposition}

\begin{proof}
This follows immediately from the definitions.
\end{proof}

\begin{lemma}
\label{lem:no-3-space}
There is no $G$-invariant isotropic $3$-space in $\wedge^2V$.
\end{lemma}

\begin{proof}
If there is such a $3$-space, then, by Lemma \ref{lem:omnibus}, it
is of the form $L\wedge V$ or of the form $\wedge^2U$. 
Note that $G$-invariance of the $3$-space combined with the uniqueness
statements in Lemma~\ref{lem:omnibus} imply that $L$ or $U$ is
also $G$-invariant, which contradicts irreducibility of $\sigma$.
\end{proof}

The proof of the following proposition would be substantially
simpler if the action of $G$ on $\wedge^2V$ were completely
reducible. However, in the current generality this need not be true.

\begin{proposition}
\label{prop:2-subspace}
Suppose that $\sigma$ is not essentially self-dual of
symplectic type. Then $\sigma$ has a nontrivial quadratic self-twist if
and only if $\wedge^2V$ contains a $G$-invariant $2$-space.
\end{proposition}

\begin{proof}
Suppose first that $\sigma$ has a nontrivial quadratic self-twist, say 
by the character $\chi$. Let $H$ be the kernel of $\chi$ and recall
that, by Clifford theory, $\sigma\vert_{H}$ is the sum of two
$2$-dimensional subrepresentations. Let $W<V$ be the $H$-invariant
$2$-space on which one of these subrepresentations is realized. Then
it is easy to see that the $G$-translates of $\wedge^2W$ span a
$G$-invariant $2$-space in $\wedge^2V$. (The reader should compare this 
with the proof of Proposition~\ref{prop:gl2-quadratic}. Indeed, the $2$-dimensional 
$G$-invariant subspace is the induction to $G$ of the determinant character
of the representation of $H$ on $W$. Recall that if $\sigma$ is a Galois representation
that corresponds to an automorphic representation $\pi$, then 
the determinant character of $\sigma$ corresponds to the central character of $\pi$.)

Now suppose that $\wedge^2V$ contains a $G$-invariant $2$-space $P$. The
kernel of $B\vert_{P}$ is $G$-invariant and thus is either $\{0\}$ or
$P$, for the first hypothesis implies that there can be no
$G$-invariant line in $\wedge^2V$. Suppose that the kernel is $P$. Then the
$4$-space $P^{\perp}$ contains $P$ and the form $B$ and the action
of $G$ pass down to $P^{\perp}/P$. Suppose that $B$ has a
nontrivial kernel in $P^{\perp}/P$. This kernel cannot be all of
$P^{\perp}/P$, for then $P^{\perp}$ would be an isotropic $4$-space
in $\wedge^2V$. Thus the kernel must be a line in $P^{\perp}/P$ and this
kernel is necessarily $G$-invariant. The preimage of this line in
$P^{\perp}$ is an isotropic $G$-invariant $3$-space in $\wedge^2V$, contrary
to Lemma \ref{lem:no-3-space}. We conclude that $(P^{\perp}/P,B)$ is
a nondegenerate quadratic $2$-space. Such a space is isomorphic to
a hyperbolic plane and hence contains exactly two isotropic lines.
The action of $G$ on $P^{\perp}/P$ is by similitudes, hence it permutes these
lines. Taking the preimage in $P^{\perp}$, we obtain two isotropic
$3$-spaces $\Lambda_1$ and $\Lambda_2$ in $\wedge^2V$ that are permuted by $G$. By
Lemma \ref{lem:no-3-space}, these isotropic $G$-spaces cannot be
fixed by $G$ and we conclude that the stabilizer of each is a
subgroup $H$ of index two in $G$. By repeating the argument of Lemma
\ref{lem:no-3-space} with $H$ in place of $G$, we conclude that
there is either a line $L<V$ or a $3$-space $U<V$ that is
$H$-invariant. By replacing $\sigma$ by $\sigma^{\vee}$ if necessary, we may
assume that the former possibility holds. Let $g_0\in G-H$. Then the
$2$-space $L+\sigma(g_0)L$ is easily seen to be $G$-invariant, contrary
to the irreducibility of $\sigma$ and $\sigma^{\vee}$. This contradiction
finally allows us to conclude that the restriction of $B$ to $P$ is
nondegenerate.

We now repeat the argument of the previous paragraph with the space
$P$ in place of the space $P^{\perp}/P$. It yields an index two
subgroup $H$ of $G$ and two isotropic lines in $\wedge^2V$ that are fixed by
$H$. By Lemma \ref{lem:omnibus}, each of these lines has the form
$\wedge^2Q$ with $Q<V$ a $2$-space. By the uniqueness assertion from
Lemma \ref{lem:omnibus}, each of these $2$-spaces is $H$-invariant.
It now follows from Clifford theory that if $\chi$ is the nontrivial
$\{\pm 1\}$-valued character on $G$ whose kernel is $H$, then
$\chi\otimes\sigma\cong\sigma$. Thus $\sigma$ has a nontrivial quadratic
self-twist, as required.
\end{proof}

\begin{proposition}
\label{prop:3-subspace}
Suppose that $\sigma$ is neither essentially self-dual of symplectic
type nor has a nontrivial quadratic self-twist. Then $\sigma$ is
essentially self-dual of proper orthogonal type if and only if $\wedge^2V$
contains a $G$-invariant $3$-space.
\end{proposition}

\begin{proof}
Let $W < \wedge^2V$ be a $G$-invariant $3$-space. By Lemma
\ref{lem:no-3-space}, $W$ cannot be isotropic. The group $G$ acts on $W$ by
similitudes and so the kernel of $B\vert_{W}$ is $G$-invariant. We
have just observed that this kernel cannot be $W$ and, by the
hypotheses and the preceding results, it cannot be of dimension $1$
or $2$. Thus the restriction of $B$ to $W$ is nondegenerate. It
follows from Lemma \ref{lem:omnibus} that there is a nondegenerate
symmetric bilinear form $C$ on $V$ such that
$\sigma(G)\subset\mathrm{GSO}(C)$. This implies that $\sigma$ is
essentially self-dual of proper orthogonal type.

Now suppose that $\sigma$ is essentially self-dual of proper orthogonal
type, so that there is a nondegenerate bilinear form $C$ on $V$
such that $\sigma(G)\subset\mathrm{GSO}(C)$. By Lemma
\ref{lem:omnibus}, there is a $3$-space $W \subset \wedge^2V$ such that
$\rho(\mathrm{GSO}(C))$ preserves $W$. In particular, $W$ is
$G$-invariant, and the reverse implication is proved.
\end{proof}

\begin{proof}[Proof of Theorem~\ref{thm:galois-side}]
The representation $\wedge^2\sigma$ is essentially self-dual. Thus if
it has any proper nonzero $G$-invariant subspace, it necessarily
has such a subspace of dimension at most $3$. The proof follows from 
Propositions~\ref{prop:1-subspace}, \ref{prop:2-subspace} and \ref{prop:3-subspace}. 
\end{proof}

\begin{rem}
Using the above theorem it is possible to explain the seemingly strange 
exception in {\bf (iii)}($\alpha$) of Theorem~\ref{thm:main}. 
Let $H$ be a subgroup of index $2$ in a group $G$, and let $(\tau,W)$ 
be a two-dimensional irreducible ($\ell$-adic) representation of $H$. 
Let $\sigma := {\rm As}(\tau)$ be the corresponding $4$-dimensional 
representation of $G$ as in Section~\ref{sec:asai}, where `$\Gamma$' acts
via $G/H$ which switches the factors of $W\otimes W$.
Assume that $\sigma$ is irreducible.  
Since $\tau$ is $2$-dimensional, there is a symplectic form $S$ on 
$W$ that $\tau$ preserves up to similitudes. 
It is easy to see that $\sigma$ preserves $S \otimes S$ on $W \otimes W$ 
up to similitudes. In fact, 
$\sigma$ is an essentially self-dual representation of improper orthogonal 
type. Further, one can see that $\tau$ is dihedral if and only if
$\sigma$ has a nontrivial quadratic self-twist. By the above theorem, one 
concludes that $\wedge^2\sigma$ is reducible if and only if $\tau$ is dihedral. 
\end{rem}

\bigskip

\noindent\address{Department of Mathematics \\ Oklahoma State University \\ Stillwater, OK 74078, USA}

\smallskip

\noindent\email{asgari@math.okstate.edu \\ araghur@math.okstate.edu}

\end{document}